\documentclass[a4paper, 11pt]{article}
\usepackage[usenames]{color} 
\usepackage{amssymb,amsthm,amsmath,amsfonts} 
\usepackage{lipsum}
\usepackage[a4paper]{geometry}
\usepackage[latin1]{inputenc} 
\usepackage[francais, english]{babel}
\usepackage{empheq}
\usepackage[T1]{fontenc}
\usepackage{faktor}
\usepackage[all]{xy}
\usepackage{amsrefs}
\usepackage[hidelinks]{hyperref}

\theoremstyle{plain}
\newtheorem{theorem}{Theorem}[subsection]
\newtheorem{theo}{Theorem}
\newtheorem{proposition}[theorem]{Proposition}
\newtheorem{corollary}[theorem]{Corollary}
\newtheorem{lemma}[theorem]{Lemma}
\newtheorem{definition-proposition}[theorem]{Definition/Proposition}
\theoremstyle{definition}
\newtheorem{definition}[theorem]{Definition}
\theoremstyle{remark}
\newtheorem{remark}[theorem]{Remark}
\theoremstyle{example}
\newtheorem{example}[theorem]{Example}
\theoremstyle{notation}

\newcommand\C{\mathbb{C}}
\newcommand\R{\mathbb{R}}

\newcommand\h{{\mathcal{H}}}

\newcommand\p{{\text{\bfseries{P}}}}

\newcommand\Dr{\mathrm{D}}

\DeclareMathOperator*{\adm}{adm}

\newcommand\dz{\frac{\partial}{\partial z}}

\newcommand\dt{\frac{\partial}{\partial t}}

\newcommand{\Addresses}{{
  \bigskip
  \footnotesize

  \textit{E-mail address} : \texttt{robtim13@hotmail.fr}}}

\title{Moduli of Legendrian foliations and quadratic differentials in the Heisenberg group}
\author{Robin Timsit}
\date{}

\begin{document}

\maketitle
\abstract

The aim of the paper is to prove the following result concerning moduli of curve families in the Heisenberg group. Let $\Omega$ be a domain in the Heisenberg group foliated by a family $\Gamma$ of legendrian curves. Assume that there is a quadratic differential $q$ on $\Omega$ in the kernel of an operator defined in \cite{Tim2} and every curve in $\Gamma$ is a horizontal trajectory for $q$. Let $l_\Gamma : \Omega \rightarrow ]0,+\infty[$ be the function that associates to a point $p\in \Omega$, the $q$-length of the leaf containing $p$. Then, the modulus of $\Gamma$ is
\[ M_4 (\Gamma) = \int_\Omega \frac{|q|^2}{(l_\Gamma) ^4} \mathrm{d} L^3.\]

\section*{Introduction and statement of results}

Holomorphic quadratic differentials on Riemann surfaces are well-studied objects that play a central role in Teichm\"uller theory; in particular in Teichm\"uller's theorem (see \cite{Ahl ,Hub, IT}) and because they are in one-to-one correspondance with measured foliation \cite{HM}. It leads us to curve families and especially to the modulus (or its reciprocal, the extremal length) of a curve family which is a classical conformal invariant studied since the work of Gr\"otzsch and his famous length-area argument. The interplay between holomorphic quadratic differentials and moduli (or extremal lengths) of curve families was then naturally considered (see \cite{Jen, Str, Iva, Ker}). The purpose of this work is to study the interplay between quadratic differentials and moduli of curve families in the setting of $3$-dimensional spherical CR geometry. 

A holomorphic quadratic differential on a Riemann surface $X$ is a holomorphic section of the line bundle $K(X)\otimes K(X)$ where $K(X)$ is the canonical bundle of $X$. In a local holomorphic coordinate $z$, a quadratic differential writes as $q(z)\mathrm{d}z^2$ for a holomorphic function $q$ of $z$. A holomorphic quadratic differential naturally gives rise to a flat singular metric on $X$ locally given by $\sqrt{|q(z)|} |\mathrm{d}z|$. Thus, the following were defined.
\begin{itemize}
\item The $q$-length of a curve $\gamma$ is
\[ l_q (\gamma) = \int_\gamma \sqrt{|q|}.\]
\item The $q$-area of $X$ is
\[ \mathrm{Area}_q (X) = \int_X |q|.\]
\end{itemize}
We emphase that the holomorphicity of $q$ is not mandatory to define these notions. 

Associated to a holomorphic quadratic differential are two (measured) singular foliations said {\it horizontal} and {\it vertical}: a curve $\gamma$ on $X$ is called a {\it horizontal trajectory} of $q$ if $q(\gamma ') >0$ and a {\it vertical trajectory} of $q$ if $q(\gamma ') <0$. What lies behind the terminology is that $q$, away from a zero, defines a local holomorphic coordinate (called {\it natural}) $w = x+iy$ such that $q = \mathrm{d} w^2$ and $\gamma$ is a horizontal (resp. vertical) trajectory of $q$ if and only if $y(\gamma)$ (resp. $x(\gamma)$) is constant. The computation of the extremal lengths of foliations has been done in a slightly different form, see for instance the remark \cite[p.~34]{Ker} that extends a work a Jenkins \cite{Jen} and Strebel \cite{Str} (see also \cite[Theorem 3.1]{Ker}). For a precise statement of Kerckhoff's remark, see \cite[Theorem 2.21]{PT} and for a detailed proof see \cite[section 1.2]{Iva}. The precise result we will be interested in states as follow (refer to section \ref{sec1.2} for the definition of the modulus of a curve family):

\begin{theo}[Theorem \ref{th1}]
Let $D$ be a domain in $\C$ foliated by a family $\Gamma$ of horizontal trajectories of a holomorphic quadratic differential $q$ on $D$. Let $l_\Gamma : D \rightarrow ]0,+\infty[$ be the function that associates to a point $z\in D$, the $q$-length of the leaf containing $z$. Then the modulus of $\Gamma$ is
\[M_2 (\Gamma) = \int_D \frac{|q|}{(l_\Gamma)^2}\mathrm{d}L^2.\]
\end{theo}

Let us point out that the computation of extremal length (\cite {Ker}) is usually done using natural coordinates and that we don't have access to those in the setting of spherical CR geometry. Thus, we need to find another strategy to prove our result. This is why
an elementary proof (because of regularity assumptions) of Theorem 1 will be given in section \ref{sec2.1}. It relies on a decomposition of the $q$-area (see Lemma \ref{lemma1}). As stated, the purpose of this paper is to prove an analog result in the setting of ($3$-dimensional) spherical CR geometry. 

\subsubsection*{CR manifolds}
For completeness, we give (spherical/strictly pseudoconvex) CR manifolds a pretty long introduction. The interested reader can learn more about CR manifolds for instance in \cite{Jac, Bog}. The conventions used in those books are not the same as here. In \cite{Jac}, the CR structure $V$ is used as our $\overline V$ (that is its sections serve as Cauchy-Riemann operator whereas here we understand sections of $\overline V$ as Cauchy-Riemann operators). In \cite{Bog}, CR manifolds are defined without specifying the rank of the CR structure and thus, what we call CR manifolds here are called CR manifolds of hypersurface type there. 

A {\it CR manifold} is a manifold $M^{2n+1}$ endowed with (complex) rank $n$ subbundle $V$ of the complexified tangent bundle $\C TM$ satisfying $V\cap \overline V = \{0\}$ and $[V,V] \subset V$. $V$ is called the {\it CR structure} on $M$. Equivalently (though we won't use this point of view here), a CR manifold is a $2n+1$ manifold $M$ endowed with a (real) rank $2n$ distribution $\p$ and a complex structure $J$ on $\p$.
\begin{remark}\label{remark0.0.1}
\begin{itemize}
\item Notice that for a $3$-dimensional manifold, the condition $[V,V] \subset V$ is always fulfilled. 
\item Some authors omit the rank $n$ dimension of the subbundle $V$ in the definition of a CR manifold and call a manifold endowed with a rank $n$ CR structure a CR manifold of {\it hypersurface type}. 
\item Every hypersurface $M$ of $\C^{n+1}$ inherits a CR structure from the complex structure on $\C^{n+1}$ simply by taking $V = \C TM \cap T^{1,0} \C ^{n+1}$. In particular, the unit sphere $S^{2n+1}$ in $\C^{n+1}$ has such a CR structure said {\it standard}.
\end{itemize}
\end{remark}
When a new structure is introduced, we usually like to understand the maps preserving this structure. A map $f : (M, V) \rightarrow (M', V')$ between CR manifold is called a {\it CR map} if $f_\ast (V) \subset V'$. A function $f: (M,V) \rightarrow \C^{k}$ is called a {\it CR function} if $f_\ast (V) \subset T^{1,0}\C^{k}$. A {\it CR diffeomorphism} between CR manifolds is a CR map which is also a diffeomorphism. Two CR manifolds are called equivalent (or CR equivalent) if there is a CR diffeomorphism from one to the other.

 A CR manifold $(M,V)$ is called {\it strictly pseudoconvex} if $V\oplus \overline V$ is a (complex) contact structure on $M$ (that is, locally $V\oplus \overline V = \ker (\omega)$ for a $1$-form $\omega$ such that $\omega \wedge (\mathrm{d} \omega)^n$ is nowhere zero). 
Using Rumin complex \cite{Rum}, one can define quadratic differentials on strictly pseudoconvex CR $3$-manifolds \cite{Tim2} and recover the notions of trajectories, $q$-length and $q$-volume. For completeness, we give a short overview about them in section \ref{sec1.1}.

\subsubsection*{Spherical CR manifolds and the Heisenberg group}
A ($3$-dimensional) spherical CR manifold \cite{BS} is a manifold locally modeled on the unit sphere $S^3 \subset \C ^2$ endowed with its standard CR structure. Equivalently (because of a Liouville type theorem: a CR diffeomorphism between open subsets of $S^3$ is the restriction of the action of an element of $\mathrm{PU}(2,1)$ on $S^3$; see \cite[p.~224]{BS} for instance), it is $(G,X)$-manifold for $X=S^3$ and $G = \mathrm{PU}(2,1)$. Another local model for a spherical CR manifold is the Heisenberg group $\h$ which is $\C \times \R$ endowed with the group law
\[(z,t)(z',t') = (z+z', t+t' + 2 \Im (z\overline z')).\]
As for its CR structure, it is given by $V_\h = span(Z)$ where $Z$ is the complex vector field
\[ Z = \dz +i\overline z \dt.\]
It is easy to see that $V_\h \oplus \overline V_\h = \ker (\omega)$ with $\omega = \mathrm{d}t-i\overline z \mathrm{d}z + iz\mathrm{d}\overline z$. Since $\omega$ is a contact form, $(\h, V_\h)$ is a strictly pseudoconvex CR manifold and it is known to be equivalent to the unit sphere minus a point endowed with its standard CR structure. Thus, one can see a spherical CR manifold as locally modeled on $(\h, V_\h)$, that is a manifold endowed with an atlas of charts (called a {\it spherical CR atlas}) $(U_i , \varphi_i : U_i \rightarrow U_i ' \subset \h)$ such that every $\varphi_j \circ \varphi_i ^{-1}$ is a CR diffeomorphism.

Moreover, a function $f : U\subset \h \rightarrow \C$ is CR if and only if $\overline Z f = 0$ and it is easy to see that a map $g=(g_1, g_2) : U \rightarrow V$ between open subsets of $\h$ is CR if and only if both $g_1$ and $g_2 + i|g_1|^2$ are CR functions. With this angle several differential operators can be defined on quadratic differentials on spherical CR manifolds and especially one: $B_2$ that remained quite mysterious in \cite{Tim2}. The main result of this paper gives a geometric interpretation of quadratic differentials in the kernel of this $B_2$, this is the result announced in the abstract (refer to section \ref{sec1.2} for the modulus of a family of a legendrian curves in the Heisenberg group).

\begin{theo}[Theorem \ref{th2}]
Let $\Omega$ be a domain in the Heisenberg group foliated by a family $\Gamma$ of legendrian (everywhere tangent to the contact distribution) curves. Assume that there is a quadratic differential $q$ on $\Omega$ such that $B_2 q = 0$ and every curve in $\Gamma$ is a horizontal trajectory for $q$. Let $l_\Gamma : \Omega \rightarrow ]0,+\infty[$ be the function that associates to a point $p\in \Omega$, the $q$-length of the leaf containing $p$.Then, the modulus of $\Gamma$ is
\[ M_4 (\Gamma) = \int_\Omega \frac{|q|^2}{(l_\Gamma) ^4} \mathrm{d} L^3.\]
\end{theo}

If in addition we ask the leaves to have the same $q$-length $C>0$, then the modulus expresses with the $q$-volume (denoted $\mathrm{Vol}_q$, see Definition \ref{defvol}) and $C$ (see Corollary \ref{coro2}):
\[M_4 (\Gamma) = \frac{1}{C^4} \mathrm{Vol}_q(\Omega).\]
As for the proof given here of the result in the complex plane, the proof of this theorem relies on a decomposition of the $q$-volume. It is for this decomposition that we need the quadratic differential to be in the kernel of $B_2$.

The paper is organized as follow. The first section begins with a short overview about quadratic differentials in spherical CR geometry and continues with moduli of curve families in the complex plane and in the Heisenberg group. In the second section, we prove theorems 1 and 2 and we give some examples and counter examples) of application of theorem 2.

\section{Quadratic differentials and moduli of curve families}
\subsection{Quadratic differentials on strictly pseudoconvex CR manifolds}\label{sec1.1}
We first give a short overview about quadratic differentials on a CR manifold. Most of this section can be found in \cite{Tim2}.

Let $(M,V)$ be a $3$-dimensional strictly pseudoconvex CR manifold. For a canonical bundle over $M$ we want to consider $V^\ast$. In order to have natural operators on it, we would want to see $V^\ast$ as a subbundle of $\C T^\ast M$. However, this implies a choice of a complement of $V\oplus \overline V$. In order to avoid that choice, we define a line bundle $\wedge^{1,0}$ as follow.

Let $B^{1,0}$ be the subbundle of $\C T^\ast M$ of forms vanishing on $\overline V$ and let $A$ be the subbundle of forms vanishing on $\p = V\oplus \overline V$. Then, we define
\[ \wedge ^{1,0} = \faktor{B^{1,0}}{A}\]
which is a line bundle naturally isomorphic to $V^\ast$.

\begin{definition}[Quadratic differential]
A quadratic differential on a strictly pseudoconvex CR manifold is a smooth section of the line bundle $\wedge^{1,0} \otimes \wedge^{1,0}$.
\end{definition}

The first geometric data induced by a quadratic differential on a strictly pseudoconvex manifold can be defined by analogy with the Riemann surface case. However, since quadratic differentials are defined up to a contact form, we shall restrict to {\it legendrian curves} (that is curves everywhere tangent to the contact distribution).

\begin{definition}[Trajectories, length, volume]\label{defvol} 
Let $q$ be a quadratic differential on a $3$-dimensional strictly pseudoconvex CR manifold $(M,V)$. Let also $\gamma : I \rightarrow M$ be a legendrian parametrized curve on $M$
\begin{enumerate}
\item $\gamma$ is called a
\begin{itemize}
\item horizontal trajectory of $q$ if $q(\gamma '(s)) > 0$ for all $s \in I$;
\item vertical trajectory of $q$ if $q(\gamma ' (s)) < 0$ for all $s\in I$.
\end{itemize}

\item $\sqrt{|q|}$ is a length element that can be integrated along any legendrian curve. Thus, we call the $q$-length of the curve $\gamma$, the number
\[ l_q (\gamma) = \int_\gamma \sqrt{|q|}.\]

\item The $q$-volume of $M$ is
\[ \mathrm{Vol}_q (M) = \int_M |q|^2.\]

\end{enumerate}
\end{definition}

Three differential operators on quadratic differentials on spherical CR manifolds that allowed to obtain analogue of half-translation structures were defined in \cite{Tim2}. First, let us emphase that on a spherical CR manifold, one can understand quadratic differentials in terms of cocycle relations. Namely, assume that the manifold $M$ is now endowed with a spherical CR atlas $(U_i, \varphi_i : U_i \rightarrow U_i ' \subset \h)_{i\in I}$. On every open subset of $\h$, a quadratic differential is the class modulo the contact form of a form of type $f\mathrm{d}z^2$. We write such classes $[f\mathrm{d}z^2]_\omega$. Then, if $q$ is a quadratic differential on $M$, in every chart $(U_i,\varphi_i)$ we have
\[q =\varphi_i ^\ast[q_i\mathrm{d}z^2]_\omega\]
with $q_i \in C^\infty(U_i ', \C)$. Since $q$ is globally defined on $M$, if $U_i\cap U_j \neq \emptyset$ we must have
\[ [q_i\mathrm{d}z^2]_\omega = g_{j,i}^\ast [q_j\mathrm{d}z^2]_\omega \text{ on $\varphi_i(U_i\cap U_j)$ where $g_{j,i} = \varphi_j\circ\varphi_i^{-1}$.}\]
That is
\[[q_i\mathrm{d}z^2]_\omega = [(q_j\circ g_{j,i})\left(Zg_{j,i}^1\right)^2\mathrm{d}z^2]_\omega \text{ on $\varphi_i(U_i\cap U_j)$} \]
where $g_{j,i} = \varphi_j\circ\varphi_i^{-1}$  and $g_{j,i} = (g_{j,i}^1, g_{j,i}^2)$. Which legitimates the following definition.

\begin{definition}[Quadratic differential on a spherical CR manifold] 
A quadratic differential on a spherical CR manifold $M$ with spherical CR atlas $(U_i, \varphi_i)_{i\in I}$ is a collection of smooth complex valued functions $q_i \in C^{\infty}(\varphi_i (U_i))$ satisfying for every $i,j \in I$
\[q_i = (q_j\circ g_{j,i})\left(Zg_{j,i}^1\right)^2 \text{ on $\varphi_i(U_i\cap U_j)$ where $g_{j,i} = \varphi_j\circ\varphi_i^{-1}$.}\]
\end{definition}

We can now define operators on quadratic differentials on $\h$ and "lift" them using charts on the manifold. On $\h$, the operators are the following: let $q$ be a quadratic differential defined in an open subset $U$ of $\h$
\[ \Dr '_{2,\h} [q\mathrm{d}z^2]_\omega = \left(2qZ\overline Z q - Zq\overline Zq -4iqTq\right)[\mathrm{d}z^3]_\omega \otimes \omega\wedge\mathrm{d}z \text{,}\]
\[\Dr ''_{2,\h} [q\mathrm{d}z^2]_\omega = \left(2q\overline Z ^2q - \left(\overline Z q\right)^2\right)[\mathrm{d}z^3]_\omega \otimes \omega\wedge \mathrm{d}\overline z \text{ and}\]
\[B_{2,\h} [q\mathrm{d}z^2]_\omega = \left(\overline Z\left(|q|^2\right) + \overline q \overline Z q\right)\omega\wedge\mathrm{d}\overline z \otimes \omega.\]
\begin{remark}
\begin{itemize}
\item The two first operators are extensions to quadratic differentials of a splitting by Garfield and Lee \cite{GL} of Rumin operator \cite{Rum} on $(1,0)$-forms. A quadratic differential $q$ annihilating these two operators is locally, away from a zero, the square of the differential of a complex valued CR function. To be precise, if $q$ is a quadratic differential in the kernels of both $\mathrm{D}_{2, \h} '$ and $\mathrm{D}_{2, \h} ''$, then, locally, away from a zero, $q=[(Zf)^2 \mathrm{d} z^2]_\omega$ for a complex valued CR function $f$ (see \cite[Proposition 2.3.4]{Tim2}).

\item The operator $B_{2,\h}$ appeared when the following question was asked: for a complex valued CR function $f$, how can we ensure that (locally) there is a real valued function $h$ such that $(f,h)$ is a CR map ? That is, under what condition does the equation $\overline Z (h + i|f|^2)=0$ have a real valued solution $h$ ? It turns out that, assuming $f$ CR, $[(Zf)^2\mathrm{d}z^2]_\omega$ being in the kernel of $B_2$ is a necessary and sufficient condition for the equation $\overline Z (h + i|f|^2)=0$ to have a real valued solution. For details, refer to \cite{Tim2} right before Definition 2.3.5.

Thus, if a quadratic differential $q$ annihilates $\mathrm{D}_{2, \h} '$, $\mathrm{D}_{2, \h} ''$ and $B_{2,\h}$ then $q$ is locally, away from a zero, of the form $g^\ast[\mathrm{d}z^2]_\omega$ for a CR map $g$ (see \cite{Tim2} and Proposition \ref{prop105} below). 
\end{itemize}
\end{remark}
In order to define the new line bundles involved here, first denote $\wedge^2$ the subbundle of $2$-forms vanishing on $\p^2$. Now, define $\wedge^{1,1}$ the subbundle of forms in $\wedge^2$ with vanishing interior product with every vector of $V$ and $\wedge^{2,0}$ the subbundle of forms in $\wedge^2$ with vanishing interior product with every vector of $\overline V$. Then, 

\begin{definition}\label{def104}
Let $M$ be a spherical CR manifold $M$ with atlas $(U_i, \varphi_i)_i$. Then the following operators are well-defined (where $\Gamma (E)$ refers to the space of smooth sections of any line bundle $E$ over $M$):
\[\begin{array}{cccc}
\Dr_2 ' : & \Gamma\left(\left(\wedge^{1,0}\right)^{\otimes 2}\right) & \longrightarrow & \Gamma\left(\left(\wedge^{1,0}\right)^{\otimes 3} \otimes \wedge^{2,0}\right) \\
& q & \longmapsto & \varphi_i^\ast\Dr'_{2,\h}[q_i\mathrm{d}z^2]_\omega  \text{ on $U_i$ where $q=\varphi_i ^\ast[q_i\mathrm{d}z^2]_\omega$,}
\end{array}\]
\[\begin{array}{cccc}
\Dr_2 '' : & \Gamma\left(\left(\wedge^{1,0}\right)^{\otimes 2}\right) & \longrightarrow & \Gamma\left(\left(\wedge^{1,0}\right)^{\otimes 3} \otimes \wedge^{1,1}\right)  \\
& q & \longmapsto & \varphi_i^\ast\Dr''_{2,\h}[q_i\mathrm{d}z^2]_\omega  \text{ on $U_i$ where $q=\varphi_i ^\ast[q_i\mathrm{d}z^2]_\omega$ and}
\end{array}\]
\[\begin{array}{cccc}
B_2  : & \Gamma\left(\left(\wedge^{1,0}\right)^{\otimes 2}\right) & \longrightarrow & \Gamma\left(\wedge^{1,1} \otimes A\right)  \\
& q & \longmapsto & \varphi_i^\ast B_{2,\h}[q_i\mathrm{d}z^2]_\omega  \text{ on $U_i$ where $q=\varphi_i ^\ast[q_i\mathrm{d}z^2]_\omega$.}
\end{array}\]
\end{definition}

\begin{example}
\begin{itemize}
\item For an example of a quadratic differential annihilating $\mathrm{D}_2 '$ and $\mathrm{D}_2 ''$, as we said just take the square of the differential of a CR function. For instance, we will discuss later about the trajectories of the quadratic differential $\left[\frac{\overline z ^2}{(t+i|z|^2)^2} \mathrm{d}z^2\right]_\omega$ on $\h$.
\item $\left[\frac{\overline z ^2 (t^2+|z|^4)^{\frac{2}{3}}}{|z|^{\frac{8}{3}}(t+i|z|^2)^2} \mathrm{d}z^2\right]_\omega$ is an example of a quadratic differential on $\h \backslash \{0\}\times \R$ in the kernel of $B_2$. We will be interested in it again in section \ref{sec2}
\item Finally, $\left[\frac{(t-i|z|^2)^2}{(t+i|z|^2)^4} \mathrm{d}z^2\right]_\omega$ is a quadratic differential on $\h \backslash \{0\}$ annihilating $\mathrm{D}_2'$, $\mathrm{D}_2 ''$ and $B_2$
\end{itemize}

\end{example}

Those operators were defined in order to obtain the following result, giving a structure on a spherical CR manifold analog to a half-translation structure.

\begin{proposition}[Proposition 2.3.1 in \cite{Tim2}]\label{prop105}
Let $q$ be a non-zero quadratic differential on a spherical CR manifold $M$ such that $\Dr' _2 q = 0$, $\Dr''_2 q=0$ and $B_2 q =0$. Then, around every point $p\in M$ where $q(p)\neq 0$, there is a chart $(U,\varphi)$ such that $q=\varphi^\ast [\mathrm{d}z^2]_\omega$ on $U$. 

Moreover, two such charts differ only by translation and sign, that is by $(z,t) \mapsto \left(\pm z+z_0, t+t_0 + 2\Im(\pm z \overline z_0)\right)$ which is an analogue situation as half-translation structures.
\end{proposition}

\subsection{Moduli of curve families}\label{sec1.2}

First, we recall what is the modulus of a family of curves in the complex plane (see \cite{Vas} for instance). Let $U$ be an open subset of $\C$ and $\Gamma$ be a family of $C^1$ curves lying in $U$. Denote $\adm (\Gamma)$ the set of measurable Borel functions $\rho : U \rightarrow [0,\infty]$ such that
\[ \int_\gamma \rho \mathrm{d} l = \int_a ^b \rho(\gamma (s)) |\dot \gamma_1 (s)| \mathrm{d} s \ge 1\]
for every curve $\gamma : ]a,b[ \longrightarrow \Omega$ in $\Gamma$. Elements of $\adm (\Gamma)$ are called {\it densities}. Define the {\it $2$-modulus} of $\Gamma$ by
\[ M_2(\Gamma) = \underset{\rho\in \adm (\Gamma)}\inf \int_U \rho^2 \mathrm{d} L^2\]
where $\mathrm{d} L^2$ is the Lebesgue measure on $\R ^2$. A density $\rho_0$ is said {\it extremal for $\Gamma$} if
\[ M_2(\Gamma) = \int_U \rho_0 ^2 \mathrm{d} L^2.\]
\begin{remark}
\begin{itemize}
\item Sometimes, authors give another (though equivalent) definition of the modulus of a curve family. For a non negative measurable Borel function $\rho : U \rightarrow [0,+\infty]$, denote $\mathrm{L}_\rho (\Gamma) = \underset{\gamma \in \Gamma} \inf \int_\gamma \rho \mathrm{d} l$ and $\mathrm{A}_\rho (U) = \int_U \rho^2 \mathrm{d} L^2$. Then, it is clear that 
\[ M_2 (\Gamma) = \underset {\rho} \inf \frac{\mathrm{A}_\rho (U)}{(\mathrm{L}_\rho(\Gamma))^2}\]
where the infimum is taken over all non negative measurable Borel functions $\rho$ in $U$. We chose the other definition here because with it, the essential uniqueness of an extremal density is more clear (see Theorem 2.1.2 in \cite{Vas} for instance).
\item Using the same notations, the {\it extremal length} of $\Gamma$ is
\[ E(\Gamma) = \underset{\rho} \sup \frac{(\mathrm{L}_\rho (\Gamma))^2}{\mathrm{A}_\rho(U)}.\]
\end{itemize}
\end{remark}
For curves in the Heisenberg group, things are quite the same. One needs to consider legendrian curves because these are the locally rectifiable curves for the usual left invariant metric $d_\h (p,q) = \|p^{-1}q\|_\h$ with $\|(z,t)\|_\h = \left(t^2 + |z|^4\right)^{\frac{1}{4}}$. Let $\gamma = (\gamma_1 , \gamma_2) : I \longrightarrow \h$ be a $C^1$ curve. $\gamma$ is legendrian if and only if
\[ \dot \gamma_2 (s) = -2 \Im (\overline \gamma_1 (s) \dot \gamma_1 (s)) \text{ for every $s\in I$.}\]
Then, let $\Gamma$ be a family of legendrian curves in an open subset $\Omega$ of $\h$. Denote $\adm (\Gamma)$ the set of measurable Borel functions $\rho : \Omega \longrightarrow [0,\infty]$ such that
\[ \int_\gamma \rho \mathrm{d} l = \int_a ^b \rho(\gamma (s)) |\dot \gamma_1 (s)| \mathrm{d} s \ge 1\]
for every curve $\gamma : ]a,b[ \longrightarrow \Omega$ in $\Gamma$. Elements of $\adm (\Gamma)$ are called {\it densities}. Define the {\it $4$-modulus} of $\Gamma$ by
\[ M_4(\Gamma) = \underset{\rho\in \adm (\Gamma)}\inf \int_\Omega \rho^4 \mathrm{d} L^3\]
where $\mathrm{d} L^3$ is the Lebesgue measure on $\R ^3$. A density $\rho_0$ is said {\it extremal for $\Gamma$} if
\[ M_4(\Gamma) = \int_\Omega \rho_0 ^4 \mathrm{d} L^3.\]

As in the complex case, an extremal density is essentially unique (see \cite[Proposition 3.4]{BFP2} for instance). The modulus of a family of curves is a conformal invariant \cite{KR2} and a classical tool in the theory of quasiconformal maps in the Heisenberg group \cite{KR, KR2}, with various purposes such as finding extremal quasiconformal maps (see \cite{BFP, BFP2, Tim, Tan}) or for proving non-existence of quasiconformal maps between certain CR manifolds (see \cite{Kim, Min}). We will be interested in computing the modulus of a family of curves when this family is a foliation by horizontal trajectories of a quadratic differential.

\section{Modulus of a foliation by trajectories of a quadratic differential}\label{sec2}

\subsection{The result in the complex plane}\label{sec2.1}
As stated in the introduction, Theorem 1 is a local formulation of a well known (\cite[Theorem 2.21]{PT}) result dealing with the computation of the modulus (or extremal length) of a foliation in the complex plane. We give it here an elementary proof that doesn't use all the potential of conformal geometry because these are tools we don't have to our disposal in the setting of families of legendrian curves in the Heisenberg group. To be precise, we give a proof that uses only the annihilation of  the Cauchy-Riemann operator by the quadratic differential. 

\begin{theorem}\label{th1}
Let $I,J$ be intervals and $\Phi : I\times J \longrightarrow U \subset \C$ be a $C^2$-diffeomorphism. Denote
\[\Gamma = \{ \Phi(.,p) \ | \ p \in J\}.\]
Let $q$ be a holomorphic quadratic differential on $U$ such that every curve in $\Gamma$ is a horizontal trajectory of $q$. Then,
\[ M_2 (\Gamma) = \int_U \frac{|q|}{(l_\Gamma)^2} \mathrm{d}L^2\]
where, for every $z\in U$, $l_\Gamma (z) = l_q ( \Phi(.,\Phi^{-1}_2 (z)))$.
\end{theorem}

The proof is a consequence of the following lemma:
\begin{lemma}\label{lemma1}
Let $q : U\rightarrow \C$ be a holomorphic quadratic differential on a domain $U$ of $\C$. Let $I,J$ be intervals and $\Phi : I \times J \rightarrow U$ be a diffeomorphism such that every curve $\Phi(.,p)$ for $p\in J$ is a horizontal trajectory for $q$. Then, there is a nowhere vanishing function $\lambda : J \rightarrow \R$ such that
\begin{eqnarray}\label{eqlem1}
(|q|\circ \Phi)\mathrm{J}_\Phi = \sqrt{(|q| \circ \Phi)}|\partial_s \Phi|\lambda
\end{eqnarray}
where $\mathrm{J}_\Phi$ is the Jacobian determinant of $\Phi$.
\end{lemma}

\begin{remark}\label{rk1}

\begin{itemize}
\item What the lemma says is that, on $I\times J$, the $q$-area decomposes as the product of the $q$-length and a measure on $J$.
\item We will prove and use the lemma with the conclusion written in another form, easier for computations. We use here the same notations and hypotheses than in the lemma. First, we write 
\[\mu = \sqrt{(|q|\circ \Phi)}|\partial_s \Phi| = \sqrt{ (q\circ \Phi)(\partial_s \Phi)^2}.\]
We then multiply both sides of equation \ref{eqlem1} by $|\partial_s \Phi|^2$  (which is nowhere zero since $\Phi$ is a diffeomorphism) to obtain
\[ (|q|\circ \Phi) |\partial_s \Phi|^2 \mathrm{J}_\Phi = \mu |\partial_s \Phi|^2 \lambda.\]
Then, notice that $(|q|\circ \Phi) |\partial_s \Phi|^2 = \mu ^2$ and divide both sides by $\mu$ (which is nowhere zero) and we finally have the form of the conclusion we will prove:
\begin{eqnarray}\label{eqlem1.2}
\sqrt{(q\circ \Phi)(\partial_s \Phi)^2} \mathrm{J}_\phi = \lambda |\partial_s \Phi|^2.
\end{eqnarray}
\end{itemize}
\end{remark}

\begin{proof}[Proof of Theorem \ref{th1}]

Assuming the lemma for now, let $\rho \in \mathrm{adm}(\Gamma)$ be an admissible density for $\Gamma$. Then,
\[ \int_U \rho^2 \mathrm{d}L^2 = \int_J\int_I \rho^2(\Phi) |\mathrm{J}_\Phi| \mathrm{d}s\mathrm{d}p.\]
However, for every $p$ in $J$, Cauchy-Schwarz inequality gives
\[\left(\int_I \left(\rho(\Phi) |\mathrm{J}_\Phi|^{\frac{1}{2}}\right)^2 \mathrm{d}s \right)\left(\int_I \left(\frac{|\partial_s \Phi|}{|\mathrm{J}_\Phi|^{\frac{1}{2}}}\right)^2 \mathrm{d}s\right) \ge \left(\int_I \rho(\Phi) |\partial_s \Phi| \mathrm{d} s \right)^2 \underset{\text{$\rho\in \mathrm{adm}(\Gamma)$}} \ge 1.\]
Using equation \ref{eqlem1.2}, we have for every $p \in J$
\[\int_I \left(\frac{|\partial_s \Phi|}{|\mathrm{J}_\Phi|^{\frac{1}{2}}}\right)^2 \mathrm{d}s = \int_I \frac{\sqrt{(q\circ\Phi)\left(\partial_s \Phi\right)^2}}{|\lambda|}\mathrm{d}s = \frac{l_q(\Phi(.,p))}{|\lambda(p)|}.\]
Thus,
\[ \int_U \rho^2 \mathrm{d}L^2 \ge \int_J \frac{|\lambda(p)|}{l_q(\Phi(.,p))} \mathrm{d}p.\]
Since it is true for every $\rho \in \mathrm{adm}(\Gamma)$,
\[M_2 (\Gamma) \ge \int_J \frac{|\lambda (p)|}{l_q(\Phi(.,p))} \mathrm{d}p.\]
Now, let $\rho_0$ be defined for every $z$ by $\rho_0 (z)= \frac{\sqrt{|q(z)|}}{l_q(\Phi(.,\Phi^{-1}_2 (z)))}$. Then, if $\Phi(.,p_0)$ is a curve in $\Gamma$, one has
\[ \int_{\Phi(.,p_0)} \rho_0 \mathrm{d}l = \frac{1}{l_q(\Phi(.,p_0))} \int_I \sqrt{|q(\Phi(s,p_0))|}|\partial_s \Phi(s,p_0)| \mathrm{d}s =1\]
and so $\rho_0$ is admissible for $\Gamma$. Moreover,
\[\int_U \rho_0 ^2 \mathrm{d}L^2 = \int_J \int_I \frac{|q(\Phi)| |\mathrm{J}_\Phi|}{(l_q (\Phi(.,p)))^2} \mathrm{d}s\mathrm{d} p.\]
Using Lemma \ref{lemma1}, we obtain
\[\int_U \rho_0 ^2 \mathrm{d}L^2 = \int_J \int_I \frac{\sqrt{|q(\Phi)|}|\partial_s \Phi||\lambda|}{(l_q(\Phi(.,p)))^2} \mathrm{d}s\mathrm{d}p = \int_J \frac{|\lambda(p)|}{l_q(\Phi(.,p))} \mathrm{d}p.\]
Thus,
\[M_2 (\Gamma) \le \int_J \frac{|\lambda(p)|}{l_q(\Phi(.,p))}\mathrm{d}p.\]
\end{proof}

\begin{proof}[Proof of Lemma \ref{lemma1}]
As in remark \ref{rk1}, we write 
\[ \mu = \sqrt{(q \circ \Phi) \left(\partial_ s \Phi\right)^2}\]
which is a positive valued function. Moreover,
\[ \mathrm{J}_\Phi = \Im \left(\overline{\partial _s \Phi} \partial_p \Phi \right).\]
Consequently, since $\Phi$ is a diffeomorphism, $\partial_s \Phi$ is nowhere zero,
\[\mu \mathrm{J}_\Phi = \Im \left(\mu \overline{\partial _s \Phi} \partial_p \Phi \right) = |\partial _s \Phi|^2 \Im \left( \frac{\mu \partial_p \Phi}{\partial_s \Phi}\right).\]
Thus, to prove the lemma (in the form given by equation \ref{eqlem1.2}), we need to show that $\partial _ s \left( \Im \left( \frac{\mu \partial_p \Phi}{\partial_s \Phi}\right)\right) = \Im \left( \partial_s \left(\frac{\mu \partial_p \Phi}{\partial_s \Phi}\right)\right)= 0$.
For that, we first write
\[q = \left(\frac{\mu(\Phi^{-1}_1)}{\partial_s \Phi (\Phi^{-1})}\right)^2.\]
From $\partial_{\overline w } q = 0$, we obtain
\[ \partial_s \Phi \left( \partial_s \mu \partial_{\overline w} \Phi^{-1}_1 (\Phi) + \partial_p \mu \partial_{\overline w} \Phi^{-1}_2 (\Phi)\right) = \mu \left( \partial_{s,s} \Phi \partial_{\overline w} \Phi^{-1}_1(\Phi) + \partial_{s,p} \Phi \partial_{\overline w} \Phi^{-1}_2 (\Phi) \right).\]
Usual formulas for inverse functions give
\[ \partial_{\overline w} \Phi^{-1}_1 = -i \frac{\partial_p \Phi}{2 \mathrm{J}_\Phi} \text{ and }  \partial_{\overline w} \Phi^{-1}_2 = i \frac{\partial_s \Phi}{2 \mathrm{J}_\Phi}.\]
Leading to
\[ \partial_s \Phi \left(\partial_p \mu \partial_s \Phi - \partial_s \mu \partial_p \Phi\right) = \mu \left(\partial_{s,p} \Phi \partial_s \Phi - \partial_{s,s} \Phi \partial_p \Phi\right)\]
and so
\begin{eqnarray}\label{eq1.1}
\partial_s \Phi \left(\partial_s \mu \partial_p \Phi + \mu \partial_{s,p} \Phi\right) - \mu \partial_{s,s} \Phi \partial_p \Phi & = & \partial_p \mu \left(\partial_s \Phi\right)^2.
\end{eqnarray}

Now, we finally compute:
\begin{eqnarray*}
\partial_s \left(\frac{\mu \partial_p \Phi}{\partial_s \Phi}\right) & = & \frac{\partial_s \Phi \left( \partial_s \mu \partial_p \Phi + \mu \partial_{s,p} \Phi\right) - \mu \partial_p \Phi \partial_{s,s} \Phi}{\left(\partial_s \Phi\right)^2}\\
& \underset{\text{using equation \ref{eq1.1}}} = & \frac{\partial_p \mu \left(\partial_s \Phi\right)^2}{\left(\partial_s \Phi\right)^2}\\
& = & \partial_p \mu
\end{eqnarray*}
and $\partial_p \mu$ is a real valued function.
\end{proof}

If in addition to the hypotheses of the theorem we ask the curves $\Phi(.,p)$ to all have the same $q$-length $C>0$, then $l_\Gamma$ becomes constant equal to $C$ and we have the following corollary.

\begin{corollary}\label{coro1}
Let $D$ be a domain in $\C$ foliated by a family $\Gamma$ of horizontal trajectories of a holomorphic quadratic differential $q$ on $D$. Assume that each curve in $\Gamma$ has $q$-length $C>0$. Then the modulus of $\Gamma$ is
\[M_2 (\Gamma) = \frac{1}{C^2} \mathrm{Area}_q (D).\]
\end{corollary}

\subsection{The result in the Heisenberg group}
We now show Theorem 2. It's proof follows the same steps as the one of Theorem \ref{th1} with a lemma similar to Lemma \ref{lemma1}.

\begin{theorem}\label{th2}
Let $I$ be an interval, $\Lambda$ an open subset of $\R^2$ and $\Phi : I \times \Lambda \longrightarrow \Omega \subset \h$ be a $C^2$-diffeomorphism such that for every $p \in \Lambda$, $\Phi(.,p)$ is a legendrian curve. Denote
\[\Gamma = \{ \Phi(.,p) \ | \ p \in \Lambda\}.\]
Let $q$ be quadratic differential on $\Omega$ satisfying $B_2 q =0$ and such that every curve in $\Gamma$ is a horizontal trajectory of $q$. Then,
\[M_4 (\Gamma) = \int_\Omega \frac{|q|^2}{(l_\Gamma)^4}\mathrm{d} L^3\]
where, for every $(z,t)\in \Omega$, $l_\Gamma (z,t) = l_q(\Phi(.,\Phi^{-1}_2 (z,t)))$.
\end{theorem}

The main interest of that result is that the quadratic differential only needs to be in the kernel of $B_2$ and so doesn't have to be the square of the differential of (the complex part of) a CR map. The proof is a consequence of the following lemma:

\begin{lemma}\label{lemma2}
Let $q : \Omega \rightarrow \C$ be a quadratic differential on an open subset $\Omega$ of $\h$ such that $B_2 q = 0$. Let $I$ be an interval, $\Lambda$ be an open subset of $\R^2$ and $\Phi : I \times \Lambda \longrightarrow \Omega \subset \h$ be a diffeomorphism such that for every $p \in \Lambda$, $\Phi(.,p)$ is a horizontal trajectory for $q$. Then, there is a nowhere vanishing real valued function $\lambda$ on $\Lambda$ such that
\begin{eqnarray}\label{eqlem2}
(|q|^2\circ \Phi)\mathrm{J}_\Phi & = & \sqrt{(|q|\circ \Phi)}|\partial_s \Phi_1| \lambda
 \end{eqnarray}
where $\mathrm{J}_\Phi$ is the Jacobian determinant of $\Phi$ and $\Phi = (\Phi_1, \Phi_2)$.
\end{lemma}

\begin{remark}\label{rk2}
\begin{itemize}
\item As for the complex case, we can read the lemma as: on $I\times \Lambda$, the $q$-volume decomposes as the product of the $q$-length and a measure on $\Lambda$.
\item We will prove and use the lemma with the conclusion written in another form, easier for computations. We use here the same notations and hypotheses than in the lemma. First, we write 
\[\mu = \sqrt{(|q|\circ \Phi)}|\partial_s \Phi_1| = \sqrt{ (q\circ \Phi)(\partial_s \Phi_1)^2}.\]
We then multiply both sides of equation \ref{eqlem2} by $|\partial_s \Phi_1|^4$  (which is nowhere zero since $s\mapsto \Phi(s,p)$ is legendrian for every $p$ and $\Phi$ is a diffeomorphism; see the beginning of the proof of Lemma \ref{lemma2}) to obtain
\[ (|q|^2\circ \Phi) |\partial_s \Phi_1|^4 \mathrm{J}_\Phi = \mu |\partial_s \Phi_1|^4 \lambda.\]
Then, notice that $(|q|^2\circ \Phi) |\partial_s \Phi_1|^4 = \mu ^4$ and divide both sides by $\mu$ (which is nowhere zero) and we finally have the form of the conclusion we will mostly use:
\begin{eqnarray}\label{eqlem2.2}
\left(\sqrt{(q\circ \Phi)(\partial_s \Phi_1)}\right)^3 \mathrm{J}_\phi = \lambda |\partial_s \Phi_1|^4.
\end{eqnarray}
\end{itemize}
\end{remark}

\begin{proof}[Proof of Theorem \ref{th2}]
Assuming the lemma for now, let $\rho \in \mathrm{adm}(\Gamma)$ be an admissible density for $\Gamma$. Then,
\[ \int_\Omega \rho^4 \mathrm{d}L^3 = \int_\Lambda \int_I \rho^4(\Phi) |\mathrm{J}_\Phi| \mathrm{d}s\mathrm{d}L^2(p).\]
However, for every $p$ in $\Lambda$, H\"older inequality gives
\begin{eqnarray*}
\left(\int_I \left(\rho(\Phi) |\mathrm{J}_\Phi|^{\frac{1}{4}}\right)^4 \mathrm{d}s \right)\left(\int_I \left(\frac{|\partial_s \Phi_1|}{|\mathrm{J}_\Phi|^{\frac{1}{4}}}\right)^{\frac{4}{3}} \mathrm{d}s\right)^3 & \ge & \left(\int_I \rho(\Phi) |\partial_s \Phi_1| \mathrm{d} s \right)^4\\
& \ge & 1 \text{ (since $\rho\in \mathrm{adm}(\Gamma)$).}
\end{eqnarray*}
Using equation \ref{eqlem2.2}, we have for every $p \in \Lambda$
\[\left(\int_I \left(\frac{|\partial_s \Phi_1|}{|\mathrm{J}_\Phi|^{\frac{1}{4}}}\right)^{\frac{4}{3}} \mathrm{d}s\right)^3 = \left(\int_I \frac{\sqrt{(q\circ\Phi)\left(\partial_s \Phi_1\right)^2}}{|\lambda|^{\frac{1}{3}}} \mathrm{d}s\right)^3 = \frac{(l_q(\Phi(.,p)))^3}{|\lambda|}\]
for a nowhere vanishing function $\lambda : \Lambda \rightarrow \R$. Thus,
\[ \int_\Omega \rho^4 \mathrm{d}L^3 \ge \int_\Lambda \frac{|\lambda (p)|}{(l_q(\Phi(.,p)))^3} \mathrm{d}L^2(p).\]
Since it is true for every $\rho \in \mathrm{adm}(\Gamma)$,
\[M_4 (\Gamma) \ge \int_\Lambda \frac{|\lambda (p)|}{(l_q(\Phi(.,p)))^3} \mathrm{d}L^2(p).\]
Now, let $\rho_0$ being defined for every $(z,t)$ by $\rho_0 (z,t) = \frac{\sqrt{|q(z,t)|}}{l_q(\Phi(.,\Phi^{-1}_2 (z,t)))}$. Then, if $\Phi(.,p_0)$ is a curve in $\Gamma$, one has
\[ \int_{\Phi(.,p_0)} \rho_0 \mathrm{d}l= \frac{1}{l_q(\Phi(.,p_0))} \int_I \sqrt{|q(\Phi(s,p_0))|}|\partial_s \Phi_1(s,p_0)| \mathrm{d}s =1\]
and so $\rho_0$ is admissible for $\Gamma$. Moreover,
\[\int_\Omega \rho_0 ^4 \mathrm{d}L^3 = \int_\Lambda\int_I \frac{|q(\Phi)|^2|\mathrm{J}_\Phi|}{(l_q(\Phi(.,p)))^4} \mathrm{d}s \mathrm{d} L^2(p).\]
Using Lemma \ref{lemma2}, we obtain
\begin{eqnarray*}
\int_\Omega \rho_0 ^4 \mathrm{d}L^3 & = & \int_\Lambda \frac{|\lambda(p)|}{(l_q(\Phi(.,p)))^4}\left(\int_I \sqrt{|q(\Phi)|}|\partial_s \Phi_1| \mathrm{d}s\right)\mathrm{d}L^2(p)\\
& = & \int_\Lambda \frac{|\lambda(p)|}{(l_q(\Phi(.,p)))^3}\mathrm{d}L^2(p).
\end{eqnarray*}
Thus,
\[M_4 (\Gamma) \le \int_\Lambda \frac{|\lambda(p)|}{(l_q(\Phi(.,p)))^3}\mathrm{d}L^2(p)\]
which ends the proof.
\end{proof}

\begin{proof}[Proof of Lemma \ref{lemma2}]
As in remark \ref{rk2}, we write
\[ \mu = \sqrt{(q\circ \Phi)\left(\partial_s \Phi_1\right)^2}\]
which is a positive valued function. Moreover, using the fact that the curves $\Phi(.,p)$ are legendrian, we easily compute
\[\mathrm{J}_\Phi = -\Im \left(\overline{\partial_s \Phi_1} A\right)\]
with
\[A = \partial_{p_1} \Phi_2 \partial_{p_2} \Phi_1 - \partial_{p_2} \Phi_2 \partial_{p_1} \Phi_1 + 2 \Phi_1 \Im \left(\partial_{p_1} \Phi_1 \overline {\partial_{p_2} \Phi_1}\right).\]
Since $\Phi$ is a diffeomorphism, $\partial_s \Phi_1$ is nowhere zero. Thus,
\[\mu^3\frac{\mathrm{J}_\Phi}{| \partial_s \Phi_1|^4} = -\Im (B)\]
with
\[ B = \frac{\mu^3}{\partial_s \Phi_1 |\partial_s \Phi_1|^2} A.\]
Consequently, we only have to prove that $\Im \left(\partial_s B\right) = 0$. 

For that, we compute
\begin{multline*}
\partial_s B = \frac{3  \mu ^2\partial_s \mu |\partial_s \Phi_1|^2 - \mu ^3 \left( 2\overline{\partial_s \Phi_1} \partial_{s,s} \Phi_1 + \partial_s \Phi_1 \overline{\partial_{s,s} \Phi_1}\right)}{\partial_s \Phi_1 |\partial_s \Phi_1|^4} A + \frac{\mu^3 |\partial_s \Phi_1|^2 \partial_s A}{\partial_s \Phi_1 |\partial_s \Phi_1|^4} 
\end{multline*}
which leads to
\begin{multline*}
\partial_s B = \frac{\mu^2}{|\partial_s \Phi_1|^4} \left(\left(3\overline{\partial_s \Phi_1}\partial_s \mu - 2\mu \frac{\overline{\partial_s \Phi_1}}{\partial_s \Phi_1} \partial_{s,s} \Phi_1 - \mu \overline{\partial_{s,s}\Phi_1}\right)A+ \mu \overline{\partial_s \Phi_1}\partial_s A\right).
\end{multline*}
Denoting 
\begin{eqnarray}\label{dsB}
D & = &\left(3\overline{\partial_s \Phi_1}\partial_s \mu - 2\mu \frac{\overline{\partial_s \Phi_1}}{\partial_s \Phi_1} \partial_{s,s} \Phi_1 - \mu \overline{\partial_{s,s}\Phi_1}\right)A+ \mu \overline{\partial_s \Phi_1}\partial_s A,
\end{eqnarray}
we have to prove that $\Im (D) = 0$.

For that, we write
\[q = \left(\frac{\mu(\Phi^{-1}_1)}{\partial_s \Phi_1 (\Phi^{-1})}\right)^2.\]
From $B_2 q = 0$, we obtain
\begin{multline}\label{eq1}
3 |\partial_s \Phi_1|^2 \left(\partial_s \mu \overline Z \Phi^{-1}_1 (\Phi) + \partial_{p_1} \mu \overline Z \Phi^{-1}_2 (\Phi) +\partial_{p_2} \mu \overline Z \Phi^{-1}_3 (\Phi) \right)\\
= \overline Z \Phi^{-1}_1 (\Phi) \left(2 \overline{\partial_s \Phi_1} \partial_{s,s} \Phi_1 +  \partial_s \Phi_1 \overline{\partial_{s,s} \Phi_1}\right)\\
 + \overline Z \Phi^{-1}_2 (\Phi) \left(2 \overline{\partial_s \Phi_1} \partial_{s,p_1} \Phi_1 +  \partial_s \Phi_1 \overline{\partial_{s,p_1} \Phi_1}\right)\\
 + \overline Z \Phi^{-1}_3 (\Phi) \left(2 \overline{\partial_s \Phi_1} \partial_{s,p_2} \Phi_1 +  \partial_s \Phi_1 \overline{\partial_{s,p_2} \Phi_1}\right).
\end{multline}
Usual formulas for inverse functions give
\begin{eqnarray}\label{zbar1}
\overline Z \Phi^{-1}_1 (\Phi) & = & \frac{i}{2\mathrm{J}_\Phi} A \ \ (\neq 0);
\end{eqnarray}
\begin{eqnarray}\label{zbar2}
\overline Z \Phi^{-1}_2 (\Phi) & = & \frac{i\partial_s \Phi_1}{2 \mathrm{J}_\Phi} \left(\partial_{p_2} \Phi_2 + 2 \Im \left(\overline{\Phi_1} \partial_{p_2} \Phi_1\right)\right)
\end{eqnarray}
and 
\begin{eqnarray}\label{zbar3}
\overline Z \Phi^{-1}_3 (\Phi) & = & \frac{i\partial_s \Phi_1}{2 \mathrm{J}_\Phi} \left(- \partial_{p_1} \Phi_2 - 2 \Im \left(\overline{\Phi_1} \partial_{p_1} \Phi_1\right)\right).
\end{eqnarray}
Thus, using equations \ref{zbar1}, \ref{zbar2} and \ref{zbar3}, equation \ref{eq1} becomes

\begin{multline}\label{eq2}
3 \overline{\partial_s \Phi_1} \partial_s \mu A - 2\mu \frac{\overline{\partial_s \Phi_1}}{\partial_s \Phi_1} \partial_{s,s} \Phi_1 A - \mu \overline{\partial_{s,s} \Phi_1} A \\
= -3 \partial_{p_1} \mu |\partial_s \Phi_1|^2 \left( \partial_{p_2} \Phi_2  + 2 \Im (\overline{\Phi_1} \partial_{p_2} \Phi_1)\right) \text{ ($\in \R$)}\\
+ 3 \partial_{p_2} \mu |\partial_s \Phi_1|^2  \left(\partial_{p_1} \Phi_2 + 2 \Im (\overline{\Phi_1} \partial_{p_1} \Phi_1)\right) \text{ ($\in \R$)}\\
+ \mu \left(\partial_{p_2}\Phi_2 + 2\Im (\overline{\Phi_1} \partial_{p_2})\right)\left(2\overline{\partial_s \Phi_1}\partial_{s,p_1} \Phi_1 + \partial_s \Phi_1 + \overline{\partial_{s,p_1} \Phi_1}\right)\\
- \mu \left(\partial_{p_1}\Phi_2 + 2\Im (\overline{\Phi_1} \partial_{p_1})\right)\left(2\overline{\partial_s \Phi_1}\partial_{s,p_2} \Phi_1 + \partial_s \Phi_1 + \overline{\partial_{s,p_2} \Phi_1}\right).
 \end{multline}
We can now use equation \ref{eq2} in equation \ref{dsB} and take advantage of terms that are real valued to find:

\begin{multline*}
\frac{\Im (D)}{\mu} = \left(\partial_{p_2} \Phi_2 + 2 \Im \left(\overline{\Phi_1} \partial_{p_2} \Phi_1\right)\right)\Im(\overline{\partial_s \Phi_1} \partial_{s,p_1} \Phi_1) \\
- \left(\partial_{p_1} \Phi_2 + 2 \Im \left(\overline{\Phi_1} \partial_{p_1} \Phi_1\right)\right) \Im(\overline {\partial_s \Phi_1} \partial_{s,p_2} \Phi_1)+\Im (\overline{\partial_s \Phi_1} \partial_s A)
\end{multline*}
Finally, we need to show that
\begin{multline*}
\Im \left(\overline {\partial_s \Phi_1} \partial_s A \right) =  -\left(\partial_{p_2} \Phi_2 + 2 \Im \left(\overline{\Phi_1} \partial_{p_2} \Phi_1\right)\right) \Im \left(\overline{ \partial_s \Phi_1} \partial_{s,p_1} \Phi_1\right)\\
+ \left( \partial_{p_1} \Phi_2 + 2 \Im \left(\overline{\Phi_1} \partial_{p_1} \Phi_1\right)\right)\Im\left( \overline{\partial_s \Phi_1} \partial_{s,p_2} \Phi_1\right).
\end{multline*}
For that, we compute $\partial_s A$:
\begin{multline*}
\partial_s A = \partial_{s,p_1} \Phi_2 \partial_{p_2} \Phi_1 + \partial_{p_1} \Phi \partial_{s,p_2} \Phi_1 - \partial_{s,p_2} \Phi_2\partial_{p_1} \Phi_1 - \partial_{p_2} \partial_{s,p_1} \Phi_1\\
+ 2 \partial_s \Phi_1 \Im \left( \partial_{p_1} \Phi_1 \overline{\partial_{p_2} \Phi_1}\right) + 2 \Phi_1 \Im\left(\partial_{s,p_1} \Phi \overline{\partial_{p_2} \Phi_1} + \partial_{p_1} \overline{\partial_{s,p_2} \Phi_1}\right).
\end{multline*}
Since $\partial_s \Phi_2 = - 2 \Im \left(\overline{\Phi_1}\partial_s \Phi_1\right)$, we find
\begin{multline*}
\partial_s A = \partial_{s,p_1} \Phi_1 \left(- \partial_{p_2} \Phi_2 - 2 \Im\left(\overline{\Phi_1} \partial_{p_2} \Phi_1\right)\right)\\
+\partial_{s,p_2} \Phi_1 \left(\partial_{p_1} \Phi_2 + 2 \Im\left(\overline{\Phi_1} \partial_{p_1} \Phi_1\right)\right)\\
+ 4 \partial_s \Phi_1 \Im\left(\partial_{p_1}\Phi_1\overline{\partial_{p_2}\Phi_1}\right).
\end{multline*}
Consequently,
\begin{multline*}
\Im \left(\overline{\partial_s \Phi_1} \partial_s A\right) = \left(-\partial_{p_2} \Phi_2 - 2 \Im \left(\overline{\Phi_1} \partial_{p_2} \Phi_1\right)\right)\Im\left(\overline {\partial_s \Phi_1} \partial_{s,p_1} \Phi_1\right)\\
+ \left(\partial_{p_1} \Phi_2 + 2 \Im\left(\overline{\Phi_1}\partial_{p_1} \Phi_1\right)\right) \Im\left(\overline{\partial_s \Phi_1} \partial_{s,p_2} \Phi_1\right)
\end{multline*}

\end{proof}

As for the complex case, if in addition to the hypotheses of the theorem we ask the curves $\Phi(.,p)$ to all have the same $q$-length $C>0$, then $l_\Gamma$ becomes constant equal to $C$ and we have the following corollary.

\begin{corollary}\label{coro2}
Let $\Omega$ be a domain in the Heisenberg group foliated by a family $\Gamma$ of legendrian curves. Assume that there is a quadratic differential $q$ on $\Omega$ such that $B_2 q = 0$ and every curve in $\Gamma$ is a horizontal trajectory for $q$ with $q$-length $C>0$. Then, the modulus of $\Gamma$ is
\[ M_4 (\Gamma) = \frac{1}{C^4} \mathrm{Vol_q} (\Omega).\]
\end{corollary}

\subsection{Examples}
\subsubsection*{Explicit computation of moduli}
As an explicit example of application of theorem \ref{th2}, we want to discuss about the quadratic differential $q_0 = \left[\frac{\overline z^2 (t^2+|z|^4)^\frac{2}{3}}{|z|^{\frac{8}{3}}(t+i|z|^2)^2}\mathrm{d}z^2\right]_\omega$ in the annulus $A_{r} = \{ (z,t) \in \h \ | \ 1 < \|(z,t)\|_\h < r\}$ for $r>1$. The horizontal trajectories of $q_0$ are the spherical arcs:
\[ \gamma_{x,\alpha} (s) = \left( \sqrt{e^x \sin (s)} \alpha e^\frac{is}{2}, e^x \cos (s)\right) \text{ for $|\alpha|=1$ and $x \in ]0, 2\ln(r)[$}\]
with $s\in ]0,\pi[$. Denote $\Gamma = \{\gamma_{x,\alpha} \ | \ 0<x<2 \ln (r) \ \& \ |\alpha|=1\}$. Then $\Gamma$ is the foliation of $A_r \backslash (\{0\} \times \R)$ by horizontal trajectories for $q_0$ which satisfies $B_2 q_0 = 0$. For every $(s,x,\alpha)$, we easily compute
\[ q_0 (\gamma_{x,\alpha} ' (s)) = \frac{1}{4\sin ^\frac{4}{3} (s)}\]
and thus, for every $(x,\alpha)$, we have
\[l_{q_0} (\gamma_{x,\alpha}) = \frac{1}{2} \int_0 ^\pi \sin^{-\frac{2}{3}} (s) \mathrm{d}s.\]
Notice that all curves in $\Gamma$ have the same $q_0$-length which we denote $C = \frac{1}{2} \int_0 ^\pi \sin^{-\frac{2}{3}} (s) \mathrm{d}s$.

Now, the $q_0$-volume of $A_r \backslash(\{0\} \times \R)$ is easy to compute (by substitution) in terms of $l_{q_0} (\gamma_{x,\alpha})$, we find
\begin{eqnarray*}
\mathrm{Vol}_{q_0} (A_r \backslash(\{0\} \times \R)) & =.& \int_{A_r \backslash(\{0\} \times \R)} \frac{\mathrm{d}L^3}{|z|^\frac{4}{3}(t^2+|z|^4)^\frac{2}{3}}\\
& = & 2 \pi \int_0 ^{2\ln (r)} \int_0 ^\pi \frac{e^{2x}}{2e^{\frac{2x}{3}} \sin^\frac{2}{3} (s) e^{\frac{4x}{3}}} \mathrm{d}s\mathrm{d}x\\
& = & 4 \pi C \ln (r) .
\end{eqnarray*}
Then, Corollary \ref{coro2} gives the modulus of $\Gamma$:
\[M_4 (\Gamma) = \frac{4\pi \ln (r)}{C^3}.\]

The same thing can be done for the vertical trajectories of $q_0$. The theorem still applies to vertical trajectories since vertical trajectories of a quadratic differential $q$ are horizontal trajectories of $-q$. The vertical trajectories are the radii
\[\delta_{y,\beta} (s) = \left(\sqrt{e^s\sin(y)}\beta e^{-\frac{is}{2}\cot(y)}, e^s \cos (y)\right) \text{ for $|\beta|=1$ and $y \in ]0, \pi[$}\]
with $s\in ]0,2\ln (r)[$. Denote $\Delta = \{ \delta_{y,\beta} \ | \ 0<y<\pi \ \& \ |\beta| = 1\}$. $\Delta$ is the foliation of $A_r \backslash (\{0\} \times \R)$ by vertical trajectories for $q_0$. For every $(s,y,\beta)$, we easily compute
\[ q_0 (\delta_{y,\beta}'(s)) = -\frac{1}{4\sin ^\frac{4}{3} (y)}\]
and thus, for every $(y,\beta)$, $l_{q_0} (\delta_{y,\beta}) = \frac{\ln (r)}{\sin ^\frac{2}{3} (y)}$ which leads to
\[ l_\Delta (z,t) = \frac{(t^2+|z|^4)^\frac{1}{3} \ln (r)}{|z|^\frac{4}{3}}.\]
Finally, we use Theorem \ref{th2} to compute (by substitution) the modulus of $\Delta$:
\begin{eqnarray*}
M_4 (\Delta) & = & \int_{A_r} \frac{|q_0|^2}{(l_\Delta)^4}\\
& = & \frac{1}{\ln (r)^4}\int_{A_r} \frac{|z|^4}{(t^2+|z|^4)^2} \mathrm{d}L^3\\
& = & \frac{2\pi}{\ln (r)^3}\int_0 ^\pi \sin^2 (y) \mathrm{d}y\\
& = & \frac{\pi^2}{\ln (r) ^3}.
\end{eqnarray*}

\subsubsection*{General strategy}
As an end, let us give a general strategy to compute the modulus of a foliation by Legendrian curves, applying theorem \ref{th2}.\\
Let $\Phi : I\times \Lambda \rightarrow \Omega$ be a diffeomorphism. Then, the family of quadratic differentials for which the $\Phi(.,p)$ are horizontal trajectories is the family
\[ \left\{ q_f = \left[\frac{f}{(\partial_s \Phi_1)^2 \circ \Phi^{-1}} \mathrm{d} z^2\right]_\omega \ | \ f:\Omega \rightarrow ]0,+\infty[\right\}.\]
Thus, to apply theorem \ref{th2} one needs to find a positive valued function $f$ such that $B_2 q_f = 0$.

\begin{remark}
A quadratic differential with prescribed trajectories in the kernel of $B_2$ is (if it exists) unique up to a positive constant. Indeed, assume that $q$ is a quadratic differential and $q'$ are quadratic differentials with the same horizontal trajectories. We just saw that there must be a positive valued function $f$ such that $q' = fq$. If we assume moreover that both $q$ and $fq$ annihilate $B_2$. Then,
\[ 0 = B_2 (fq) = f^2 B_2 q + \left(|q|^2 \left( \overline Z(f^2) + f\overline Z f\right)\right)\omega\wedge\mathrm{d}\overline z \otimes \omega = \left(3 f |q|^2 \overline Z f\right)\omega\wedge\mathrm{d}\overline z \otimes \omega.\]
Thus, $f$ is a positive CR function, hence constant.
 \end{remark}

\Addresses

\end{document}